\documentclass[11pt]{amsart}

\usepackage{amssymb,amsfonts,amsmath}
\usepackage{scalerel}
\usepackage{makecell}
\usepackage{amsthm, stmaryrd}
\usepackage{graphicx}
\usepackage[all,arc]{xy}
\usepackage{hyperref}
\hypersetup{colorlinks,allcolors=violet}
\usepackage{enumerate}
\usepackage{bbm}
\usepackage{dsfont}
\usepackage{mathrsfs}
\usepackage{tikz-cd}
\usetikzlibrary{shapes}
\usepackage{import}
\usepackage{float}
\restylefloat{table}
\usepackage{multirow}
\usepackage{pdflscape}
\usepackage{listofitems}
\usepackage{standalone}
\usepackage[T1]{fontenc}
\usepackage[toc,page]{appendix}
\usepackage{booktabs}
\usepackage{array}
\usepackage[margin=1in]{geometry}
\usepackage{mathtools}
\usetikzlibrary{decorations.pathmorphing}
\usepackage{mathrsfs}
\usepackage{cancel}
\usepackage{relsize}
\newtheorem{thm}{Theorem}[section]
\newtheorem*{thm*}{Theorem}
\newtheorem*{metathm*}{Meta Theorem}

\newtheorem*{setup*}{Setup}

\newtheorem{innercustomthm}{Theorem}
\newenvironment{customthm}[1]
  {\renewcommand\theinnercustomthm{#1}\innercustomthm}
  {\endinnercustomthm}

\newtheorem{innercustomcor}{Corollary}
\newenvironment{customcor}[1]
  {\renewcommand\theinnercustomcor{#1}\innercustomcor}
  {\endinnercustomcor}

\newtheorem{cor}[thm]{Corollary}
\newtheorem{prop}[thm]{Proposition}
\newtheorem{lem}[thm]{Lemma}

\theoremstyle{definition}
\newtheorem{defn}[thm]{Definition}

\newtheorem{rem}[thm]{Remark}

\newtheorem*{thm1.2}{\textrm{Theorem 1.2}}

\theoremstyle{remark}
\newcommand{\Mbar}{\overline{\mathcal{M}}}

\newcommand{\M}{\mathcal{M}}

\newtheorem{sch}[thm]{Scholium}

\newcommand{\floor}[1]{\left\lfloor #1 \right\rfloor}

\newcommand{\Z}{\mathbb{Z}}

\newcommand{\QQ}{\mathbb{Q}}

\renewcommand{\P}{\mathbb{P}}
\newcommand{\bbS}{\mathbb{S}}

\newcommand{\Aut}{\operatorname{Aut}}

\newcommand{\ch}{\operatorname{ch}}

\newcommand{\cat}[1]{\mathsf{#1}}

\newcommand{\nocontentsline}[3]{}
\let\origcontentsline\addcontentsline
\newcommand\stoptoc{\let\addcontentsline\nocontentsline}
\newcommand\resumetoc{\let\addcontentsline\origcontentsline}
\def\P{\mathbb{P}}

\def\Z{\mathbb{Z}}

\def\calC{\mathcal{C}}

\def\calL{\mathcal{L}}
\def\calM{\mathcal{M}}

\def\calS{\mathcal{S}}
\def\calT{\mathcal{T}}

\def\M{\mathcal{M}}

\newcommand{\Res}{\operatorname{Res}}

\newcommand{\Exp}{\operatorname{Exp}}
\newcommand{\Log}{\operatorname{Log}}

\newcommand{\ve}{\varepsilon}

\newcommand{\triv}{\operatorname{triv}}

\newcommand{\Pic}{\mathrm{Pic}}
\newcommand\cycle[2][\,]{%
  \readlist\thecycle{#2}%
  (\foreachitem\i\in\thecycle{\ifnum\icnt=1\else#1\fi\i})%
}

\newcommand{\gr}{\mathrm{gr}}

\makeatletter
\let\c@equation\c@thm
\makeatother
\numberwithin{equation}{section}

\graphicspath{ {images/} }

\title{Euler characteristics of the universal Picard stack}

\author[S. Kannan]{Siddarth Kannan}\address{Department of Mathematics, Massachusetts Institute of Technology}
\email{\url{spkannan@mit.edu}}

\begin{document}
\maketitle

\begin{abstract}
    We study $\bbS_n$-equivariant weight-graded and topological Euler characteristics of the universal Picard stack $\Pic_{g, n}^d \to \M_{g, n}$ of degree-$d$ line bundles over $\M_{g, n}$. We prove that in the weight-zero and topological cases, the generating function for Euler characteristics of $\Pic_{g, n}^d$ is obtained from the corresponding one for $\M_{g, n}$ by an extremely simple combinatorial transformation. This lets us deduce closed formulas for the two generating functions, taking as input the Chan--Faber--Galatius--Payne formula in the weight-zero case and Gorsky's formula in the topological case. As an immediate corollary, we obtain closed formulas for the weight-zero and topological Euler characteristics of $\Pic^d_g$. Our weight-zero calculations follow from a general result passing from the weight-graded Euler characteristics of $\M_{g, n}$ to those of $\Pic_{g,n}^d$.
\end{abstract}

\section{Introduction}

Fix integers $g \geq 1$ and $n,d \geq 0$ such that $(g, n) \neq (1, 0)$. Consider the universal Picard stack
\[\Pic^{d}_{g, n} \to \M_{g, n} \]
of degree-$d$ line bundles over the moduli space of smooth $n$-pointed curves of genus $g$. It is a smooth Deligne--Mumford stack of dimension $4g - 3 + n$ which is proper over $\M_{g, n}$. Let
\[\gr_k^WH^\star_c(\Pic_{g, n}^d, \QQ) :=  W_kH^\star_c(\Pic^{d}_{g, n}, \QQ)/W_{k-1}H^\star_c(\Pic^{d}_{g, n}, \QQ) \]
be the $k$th weight-graded piece of the compactly-supported cohomology. We are interested in the $\bbS_n$-equivariant weight-$k$ compactly-supported Euler characteristic
\[ \chi_k^{\bbS_n}(\Pic_{g, n}^d) = \sum_{i}(-1)^i \ch_n\left(\gr^W_kH^i_c(\Pic_{g, n}^d, \QQ)\right) \in\hat{\Lambda} , \]
where
\[\hat\Lambda := \QQ[\![p_1, p_2, \ldots]\!]\]
is the ring of degree-completed symmetric functions and $\ch_n(V)$ denotes the Frobenius characteristic of a finite-dimensional $\bbS_n$-representation $V$. The weight-$k$ Euler characteristic of $\Pic^d_{g, n}$ does not depend on $d$ \cite[Lemma 4.3]{VirtualHodge}, so we will write
\[\Pic_{g, n} := \Pic^0_{g, n} \]
for the universal Jacobian, and all of our calculations on $\Pic_{g, n}$ will apply to $\Pic_{g, n}^d$ as well. For any $k$, define generating functions
\begin{equation}\label{eqn:JandZdefn} \cat{J}_g^k := \sum_{n} \chi_k^{\bbS_n}(\Pic_{g, n}) \quad \mbox{and}\quad\cat{Z}_g^k := \sum_{n} \chi_k^{\bbS_n}(\M_{g, n});
\end{equation} the sums are taken over $n \geq 1$ if $g = 1$ and over $n \geq 0$ otherwise. Our main result Theorem \ref{thm:general_transform} passes from $\cat{Z}_g^k$ to $\cat{J}_g^k$. It implies a closed formula for $\cat{J}_g^0$ (Corollary \ref{cor:Jgformula}), and a similar technique yields a closed formula for topological Euler characteristics (Theorem \ref{thm:top_chi}). Sample calculations are included at the end of the paper in Tables \ref{table:wt0data} and \ref{table:Jgtop}. 

In order to state Theorem \ref{thm:general_transform}, we define a $\QQ$-algebra homomorphism $\calT: \hat{\Lambda} \to \hat{\Lambda}[\![x]\!]$ by
\begin{equation}\label{eqn:Transform_defn}
\calT: p_n \mapsto (p_n + x^n)/(1 - x^n).
\end{equation}
\begin{customthm}{1}\label{thm:general_transform}
For any $k$,
\[ \lim_{x \to 1}(1-x)\calT(\cat{Z}_g^k) = \sum_{i = 0}^{\floor{k/2}} \cat{J}_{g}^{k - 2i}. \]
The limit is interpreted as multiplying by $(1 -x)$ and then evaluating at $x= 1$.
\end{customthm}
As a consequence, we obtain an explicit formula for $\cat{J}_g^0$ as a simple transformation of the elegant expression for $\cat{Z}_{g}^0$ proved in \cite{CFGP} (there, the notation $z_g$ is used for $\cat{Z}_g^0$). We set $P_k := 1 + p_k \in \hat{\Lambda}$.
\begin{customcor}{2}\label{cor:Jgformula}
Suppose $g \geq 2$. The generating function $\cat{J}_g^0$ is obtained from the expression for $\cat{Z}_{g}^0$ in \cite[Theorem 1.1]{CFGP} by restricting the sum to those Laurent monomials in the $P_j$ whose exponents sum to $1$, and then making the substitution $P_j \mapsto P_j/j$ for all $j$. Specifically,
\[ \cat{J}_g^0 = \sum_{k, m, s, a, d} \frac{(-1)^{k}m^{k-1} (k-1)!}{P_m^k}  \prod_{i = 1}^{s} \frac{\mu(m/d_i)^{a_i} P_{d_i}^{a_i}}{d_i^{a_i}a_i!}. \]
where the sum is over integers $k, m, s > 0$ and $s$-tuples of positive integers $a = (a_1, \ldots, a_s)$ and $d = (d_1, \ldots, d_s)$ with $\gcd(d_1, \ldots, d_s) = 1$ and $d_i \mid m$ for each $i$, such that $0 < d_1 < \cdots < d_s < m$, $\sum_i a_i = k + 1$, and $\sum_i a_i d_i = km + 1 - g$.
\end{customcor}
 It is rather non-obvious to me \textit{a priori} that such a simple transformation should calculate $\cat{J}_g^0$. If we say that $P_j$ has degree $j$, Corollary \ref{cor:Jgformula} implies that $\cat{J}_{g}^0$ is a homogeneous Laurent polynomial in the $P_j$'s of degree $1 - g$, as is true for $\cat{Z}_g^0$. See Table \ref{table:wt0data} for some values of $\cat{J}_g^0$.

A similar phenomenon occurs for the \textit{topological} Euler characteristics. Write $\chi^{\bbS_n}_{top}$ for the topological Euler characteristic, considered as a virtual $\bbS_n$-representation in $\hat{\Lambda}$. Define 
\[\cat{J}_g^{top}:= \sum_{n} \chi^{\bbS_n}_{top}(\Pic_{g, n})\quad \mbox{and}\quad\cat{Z}_g^{top} := \sum_{n} \chi^{\bbS_n}_{top}(\M_{g, n}); \]
the sums are over the same $n$ as in (\ref{eqn:JandZdefn}).
When $g \geq 2$, a beautiful formula for $\cat{Z}_g^{top}$ was obtained by E. Gorsky \cite[Theorem 3.8]{Gorsky}. In his notation, $1 + p_j t^j$ corresponds to $P_j$. We find that another elementary combinatorial transformation of $\cat{Z}_g^{top}$ gives $\cat{J}_g^{top}$. In this case we focus on monomials of total exponent $2$.

\begin{customthm}{3}\label{thm:top_chi}
Suppose $g \geq 2$. The generating function $\cat{J}_g^{top}$ is obtained from Gorsky's formula \cite[Theorem 3.8]{Gorsky} for $\cat{Z}_g^{top}$ by restricting the sum to Laurent monomials in the $P_j$ whose exponents sum to $2$, and then making the substitution $P_j \mapsto P_j/j$ for all $j$. Explicitly,
    \[ \cat{J}_g^{top} = \sum_{r = 2}^{4g + 2} \sum_{s = 3}^{2g + 2} \sum_{\substack{{k_1 + \ldots + k_{r - 1} = s}\\ {\sum_i ik_i + r(2 - s)  = 2 - 2g}\\ {k_i \geq 0 \,\forall\, i}\\{i \mid r \text{ if }k_i > 0}\\{\mathrm{gcd}(i \mid k_i > 0) = 1}}}\left( (-1)^{s - 3}(s-3)! \frac{N(r;k_1, \ldots, k_{r-1})}{\prod_{i = 1}^{r - 1}i^{k_i}k_i!} \cdot r^{s - 3} \right) P_r^{2 - s} \prod_{i = 1}^{r - 1}P_{i}^{k_i}, \]
where
\[N(r; k_1, \ldots, k_{r - 1}) = \frac{1}{r} \sum_{d \mid r} \varphi(d) \prod_{i = 1}^{r-1} \left(\frac{\mu (d/\gcd(d, i)) \cdot\varphi(r/i)}{\varphi(d/\gcd(d, i))}\right)^{k_i}. \]
\end{customthm}
Just like $\cat{Z}_g^{top}$, the expression for $\cat{J}_{g}^{top}$ is a Laurent polynomial in the $P_j$'s which is homogeneous of degree $2 - 2g$. See Table \ref{table:Jgtop} for some values of $\cat{J}_g^{top}$. Analogous formulas for $\cat{J}_1^0$ and $\cat{J}_1^{top}$ can also be obtained, though we do not state them explicitly here. See \cite[Proposition 1.5]{CFGP} for the weight-zero case, and note that \cite[Equation (5.5)]{GetzlerMHM} may be specialized to the topological case.

From the definitions of the relevant generating functions, the preceding results give formulas for the weight-zero and topological Euler characteristics of $\Pic_g$.
\begin{customcor}{4}\label{cor:top_chi_pic}
For $g \geq 2$, closed formulas for $\chi_0(\Pic_g)$ and $\chi_{top}(\Pic_g)$ are obtained from $\cat{J}_g^0$ and $\cat{J}_g^{top}$ respectively by the substitution $P_j \to 1$ for all $j$.
\end{customcor}

\subsection{Remarks on the formulas}\label{rem:formula_shape} In both the weight-zero and topological cases, there is a straightforward geometric interpretation for restricting to Laurent monomials with fixed sums of exponents.

In Gorsky's formula for $\cat{Z}_g^{top}$, the coefficient of each Laurent monomial is the orbifold Euler characteristic of a moduli space of pairs $(C, \tau)$ where $C \in \M_g$ and $\tau \in \Aut(C)$ are chosen such that the branching of the cover $C \to C/\tau$ is fixed. Laurent monomials whose exponents sum to $2$ correspond to covers where $C/\tau$ has genus zero. In the formula for $\cat{Z}_g^0$ in \cite{CFGP}, the coefficient of each Laurent monomial is a certain weighted count of ``orbigraphs,'' which are graph-theoretic avatars of Gorsky's cyclic coverings; see \cite[\S 4]{CFGP}. The monomials whose exponents sum to $1$ are those where the corresponding orbigraphs are trees. That passing to the Picard stack has to do with restricting to covers of genus-zero objects in both the topological and weight-zero cases merits further exploration. I also do not have a geometric way of interpreting the substitution $P_j \to P_j/j$.

Note that for $n > 2g + 2$, the map $\Pic_{g, n}^d \to \M_{g, n}$ is an abelian fibration of varieties, which means $\chi_{top}(\Pic_{g, n}^d) = 0$ in this range. This property is manifest in Theorem \ref{thm:top_chi}, since the generating function  $\sum_n \chi_{top}(\Pic_{g, n}) t^n/n!$ for numerical Euler characteristics is obtained by sending $P_1 \to (1 + t)$ and $P_j \to 1$ for $j > 0$. One verifies that there are no negative powers of $P_1$ appearing in the formula and that the maximum power of $P_1$ which appears is $2g + 2$, which then implies the vanishing statement. However, this vanishing does not hold $\bbS_n$-equivariantly: the invariant $\chi_{top}^{\bbS_n}$ is still interesting even for $n > 2g + 2$. 

More surprising to me is that an analogous property holds for the numerical weight-zero Euler characteristic $\chi_0(\Pic_{g, n})$. Via a similar analysis of the monomials appearing in the formula, we find that it vanishes for $n > g + 1$, but not $\bbS_n$-equivariantly. It would be interesting to find an explanation for this property which comes from tropical geometry; I expect that it is related to tropical Jacobians of metric graphs \cite{AbreuPaciniSkeleton}.

\subsection{Techniques} 
Let $\calL$ be a Poincar\'e line bundle on $\Pic_{g, n}^m \times_{\M_{g, n}} \calC_{g, n}$, where
\[ \calC_{g, n} \to \M_{g, n} \]
is the universal curve. Let
\[\pi: \Pic_{g, n}^m \times_{\M_{g, n}} \calC_{g, n} \to \Pic_{g, n}^m \]
be the projection. When $m > 2g - 2$, Riemann--Roch implies that $\pi_*\calL$ is a vector bundle on $\Pic_{g, n}^m$. For such $m$, we define
\begin{equation}\label{defn:Sgn}
\calS_{g, n}^m := \P(\pi_*\calL).
\end{equation}
The projective bundle formula for Deligne--Mumford stacks \cite[Theorem 3.4]{motivesDM} then determines the rational cohomology of $\calS_{g, n}^m$ in terms of that of $\Pic_{g, n}^m$.
\begin{prop}\label{prop:sym_to_pic_weights}
    For $m \geq 2g - 1$ and for any $k$, we have
    \[ \gr^W_k H^j_c(\calS_{g, n}^{m},\QQ) \cong \bigoplus_{i = 0}^{m - g} \gr^W_{k - 2i}H_c^{j - 2i}(\Pic_{g, n}^m,\QQ)  \]
    for any $k \geq 0$. In particular,
    \[\chi_k^{\bbS_n}(\calS_{g, n}^m) = \sum_{i = 0}^{m - g} \chi_{k - 2i}^{\bbS_n}(\Pic_{g, n}). \]
\end{prop}

For topological Euler characteristics, we have
\begin{equation}\label{eqn:sym_to_pic_top}
\chi^{\bbS_n}_{top}(\calS_{g, n}^m) = \chi^{\bbS_{n}}_{top}(\Pic_{g, n}) \cdot (m - g + 1)
\end{equation}
for $m \geq 2g - 1$.
By Proposition \ref{prop:sym_to_pic_weights} in the weight-graded case and (\ref{eqn:sym_to_pic_top}) in the topological case, we are led to study Euler characteristics of the spaces $\calS_{g, n}^m$. 
The coarse moduli space of $\calS_{g, n}^m$ is isomorphic to the coarse moduli space of the relative symmetric power $[\calC_{g, n}^m /\bbS_m]$, where $\calC_{g, n}^m$ is the $m$th fibered power of $\calC_{g, n}$ over $\M_{g, n}$. Since taking coarse spaces of Deligne--Mumford stacks does not affect rational cohomology, we will use the relative symmetric power for calculations.

To calculate with relative symmetric powers, we will start with a generating function for Euler characteristics of $\M_{g, n}$, and then make a series of transformations using symmetric function theory. Each step has an interpretation in terms of allowing or forbidding groups of marked points to coincide, or in terms of taking symmetric group quotients. When these transformations are composed, we are led to the algebra map $\calT$ of (\ref{eqn:Transform_defn}). The geometry behind these transformations was studied in \cite{KSY2024}, in the context of Hassett's moduli spaces of weighted pointed stable curves \cite{Hassett}.

% The basic geometry at hand is that of relative symmetric powers over $\M_{g}$. If \[\calC_{g, n} \to \calM_{g, n}\]
% is the universal curve, we set
% \[ \calC_{g, n}^{m} \to \M_{g, n} \]
% for the $m$th fibered power of $\calC_{g, n}$ over $\M_{g, n}$. We define $\calS_{g, n}^{m} :=  [\calC_{g, n}^{m} /\bbS_m]$, and obtain a natural morphism
% \begin{equation}\label{eqn:sym-to-pic}
%     \calS_{g, n}^{m} \to \Pic_{g, n}^m.
% \end{equation}
% On the level of coarse moduli spaces (\ref{eqn:sym-to-pic}) is a projective bundle of relative dimension $m - g$ once $m \geq 2g - 1$, by Riemann--Roch. A basic application of the projective bundle formula then yields the following proposition.
\subsection{Related work}
The calculation of topological, weight-graded, and orbifold Euler characteristics of moduli spaces of curves and their variants is a fundamental problem in algebraic geometry. See e.g. \cite{HarerZagier, BiniHarer, GetzlerKapranov, GetzlerMHM, GetzlerGenus2, PaganiTommasi, PW2, IntersectionHZ, genusonechar, BCK, EulerCharAbelian, genusoneDREuler} and the references therein. 

In \cite{VirtualHodge}, T. Song and I studied the weight-graded Euler characteristics characteristics of $\Pic_{g, n}$, towards understanding those of $\M_{g, n}(\P^r, d)$. In particular, Proposition E in loc. cit. determines the generating function for $\bbS_n$-equivariant Hodge--Deligne polynomials of $\Pic_{g, n}$ in terms of the corresponding generating function for $\M_{g, n}$. It is not clear how to extract the simple formulas in the present work directly from \cite[Proposition E]{VirtualHodge}, and the calculation herein is entirely self-contained. The geometric insight used here is the same as \cite{VirtualHodge}: it is the algebraic approach, particularly the definition and analysis of the operator $\calT$ of (\ref{eqn:Transform_defn}), that is different.

Recently, S. Wood has calculated the \textit{orbifold} Euler characteristics of relative compactified Jacobians over $\Mbar_{g, n}$ \cite{WoodOrbifold}. Her formula reduces to the Euler characteristic of moduli spaces of stable curves of genus zero, which fits into the general theme discussed in Remark \ref{rem:formula_shape}.

The generating function $\cat{Z}_g^{2}$ was computed by Payne--Willwacher \cite{PW2}. It is conceivable that the limit in Theorem \ref{thm:general_transform} can be evaluated to compute $\cat{J}_g^2$ as well, but we do not pursue this here. In their work and in \cite{CFGP}, a crucial input is an interpretation of the weight-graded pieces of the cohomology of $\M_{g, n}$ in terms of the cohomology of certain graph complexes, coming from Deligne's weight spectral sequence for the normal crossings compactification $\M_{g, n} \subset \Mbar_{g, n}$ as in \cite{GetzlerKapranov}. I only need graph complexes indirectly for my calculations, in the sense that Corollary \ref{cor:Jgformula} depends on the formula in \cite{CFGP}, which in turn depends on a graph complex defined in \cite{cgp}. T. Song and I are studying graph complexes computing the weight-graded pieces of the cohomology of $\Pic_{g, n}$ in ongoing work.

\subsection*{Acknowledgments}
I am indebted to Terry Song for many inspiring discussions. I also thank Dhruv Ranganathan and Terry Song for thoughtful comments on this article. I am supported by NSF DMS-2401850.

\section{Transforms of symmetric functions}
Let $\mathrm{Cl}_{\QQ}(\bbS_n)$ denote the vector space of $\QQ$-valued class functions on $\bbS_n$. There is a natural identification
\[\hat{\Lambda} \cong \prod_{n \geq 0} \mathrm{Cl}_\QQ(\bbS_n) \]
given by the \textit{Frobenius characteristic}. The Frobenius characteristic of an $\bbS_n$-representation $V$ is defined by
\[ \ch_n(V) = \frac{1}{n!}\sum_{\sigma \in \mathbb{S}_n} \mathrm{Tr}(\sigma|V) p_1^{k_1(\sigma)}\cdots p_n^{k_n(\sigma)}, \]
where $k_i(\sigma)$ denotes the number of $i$-cycles in $\sigma$. All of the symmetric function theory we use is developed in Macdonald's foundational text \cite{Macdonald}.

There are two basic specializations of symmetric functions that are worth highlighting. First, we have the rank homomorphism
\[ \mathrm{rk}: \hat{\Lambda} \to \QQ[\![x]\!] \]
given by $p_1 \mapsto x$ and $p_k \mapsto 0$ for $k > 1$. Under this specialization, we have
\[\mathrm{rk}(\ch_n(V)) = \dim(V) \frac{x^n}{n!}. \]
We may also define
\begin{equation}\label{eqn:inv_definition}
\mathrm{inv}: \hat{\Lambda} \to \QQ[\![x]\!]
\end{equation}
by $p_k \mapsto x^k$ for all $k > 0$. Under this map, we have
\begin{equation}\label{eqn:inv_property_rep}
    \mathrm{inv}(\ch_n(V)) = \dim \left(\mathrm{Hom}_{\bbS_n}(\triv_{\bbS_n}, V)\right)x^n.
\end{equation}
This property holds for virtual representations as well, so for any $\bbS_n$-variety $X$ and integer $k$ we have
\begin{equation}\label{eqn:inv_property_space}
    \mathrm{inv}(\chi_k^{\bbS_n}(X)) = \chi_k([X/\bbS_n])x^n.
\end{equation}
\subsection{Bisymmetric functions}
We will also make use of the ring of bisymmetric functions
\[ \hat{\Lambda} \otimes \hat{\Lambda} = \QQ[\![p_1, p_2, \ldots, q_1, q_2, \ldots  ]\!]\]
where $\otimes$ means the completed tensor product, and we have identified $p_i = p_i \otimes 1$ and $q_i = 1 \otimes p_i$.

There is a coproduct
\[ \Delta : \hat{\Lambda} \to \hat{\Lambda} \otimes \hat{\Lambda}, \]
determined by $\Delta(p_i) = p_i + q_i$ for all $i$. Using the natural identification
\[ \hat{\Lambda} \otimes \hat{\Lambda} = \prod_{m,n\geq 0} \mathrm{Cl}_{\QQ}(\bbS_m \times \bbS_n), \]
we have
\[    \Delta(\ch_n(V)) = \sum_{k = 0}^n \ch_{k, n-k}\left(\Res_{\bbS_k \times \bbS_{n - k}}^{\bbS_n }V \right),
\]
where $\ch_{m,n}$ denotes the Frobenius characteristic of an $\bbS_m \times \bbS_n$-representation, which is an element of $\hat{\Lambda} \otimes \hat{\Lambda}$ defined analogously to the case of $\bbS_n$-representations. If $X$ is an $\bbS_n$-variety, we obtain
\begin{equation}\label{eqn:coproduct_chi}
\Delta \chi_k^{\bbS_n}(X) = \sum_{m = 0}^{n} \chi_k^{\bbS_{m} \times \bbS_{n-m}}(X).
\end{equation}
We will also make use of the bisymmetric versions of the map $\mathrm{inv}$. Now there are two maps
\[ \mathrm{inv}_1, \mathrm{inv}_2 : \hat{\Lambda} \otimes \hat{\Lambda} \to \hat{\Lambda}[\![x]\!], \]
determined by
\[\mathrm{inv}_1: p_i \mapsto x^i;\, q_i \mapsto p_i \]
and
\[\mathrm{inv}_2: p_i \mapsto p_i;\, q_i \mapsto x^i. \]
If $X$ has an $\bbS_m \times \bbS_n$-action, the analogues of (\ref{eqn:inv_property_space}) are straightforward:
\begin{equation}\label{eqn:bisym_inv_space_prop}
\mathrm{inv}_1(\chi_k^{\bbS_m \times \bbS_n}(X) ) = \chi_k^{\bbS_n}([X/\bbS_m]) x^m\quad \mbox{and} \quad\mathrm{inv}_2(\chi_k^{\bbS_m \times \bbS_n}(X) ) = \chi_k^{\bbS_m}([X/\bbS_n]) x^n. 
\end{equation}
The following lemma is straightforward from the definition of $\Delta$.
\begin{lem}\label{lem:inv_independence}
    For any $f \in \hat{\Lambda}$,
    \[ \mathrm{inv}_1(\Delta(f)) = \mathrm{inv}_2(\Delta(f)). \]
\end{lem}

\begin{defn}\label{defn:invDelta}
In view of Lemma \ref{lem:inv_independence}, we define
\[ \mathrm{inv}\Delta : \hat{\Lambda }\to \hat{\Lambda}[\![ x ]\!] \]
by
\[ \mathrm{inv}\Delta(f) = \mathrm{inv}_1(\Delta(f)) = \mathrm{inv_2}(\Delta(f)). \]
\end{defn}
For any $\bbS_n$-variety $X$, we obtain
\begin{equation}\label{eqn:invDelta_space}
    \mathrm{inv}\Delta(\chi_k^{\bbS_n}(X)) = \sum_{m = 0}^{n} \chi_k^{\bbS_{n - m}}([X/\bbS_m])x^m.
\end{equation}
\subsection{Plethysm}
Let
\[ \Lambda := \QQ[p_1, p_2, \ldots] \]
be the ordinary ring of symmetric functions. Recall that \textit{plethysm} is the associative operation
\[ \Lambda \times \Lambda \to \Lambda, \]
denoted by $(f, g) \mapsto f \circ g$, characterized uniquely by the following properties:
\begin{enumerate}
    \item for all $g \in \Lambda$, the map $\Lambda \to \Lambda$ by
    \[ f \mapsto f \circ g \]
    is a $\QQ$-algebra homomorphism;
    \item for all $k > 0$, the map $\Lambda \to \Lambda$ by
    \[f \mapsto p_k \circ f \]
    is a $\QQ$-algebra homomorphism;
    \item for any $f  = f(p_1, p_2, \ldots) \in \Lambda$ and $k > 0$, we have
    \[ p_k \circ f = f \circ p_k = f(p_{k}, p_{2k}, \ldots). \]
\end{enumerate}
Plethysm extends to \[\hat{\Lambda} \times \hat{\Lambda}_+ \to \hat{\Lambda} \] where $ \hat{\Lambda}_+ \subset \hat{\Lambda}$ is the subspace of symmetric functions with vanishing degree-$0$ term.
\begin{defn}
    For $f \in \hat{\Lambda}_+$ we define the plethystic exponential of $f$ by
    \[ \Exp(f) := \sum_{n > 0} h_n \circ f, \]
    where
    \[h_n := \ch_n(\triv_{\bbS_n}) \in \hat{\Lambda}.\]
\end{defn}
This version of $\Exp$ differs from the convention in \cite[(8.4)]{GetzlerKapranov}, where the sum includes $h_0 = 1$. However, it is consistent with the convention in \cite[\S 3]{GetzlerPandharipande}. With this convention, we obtain the following lemma from \cite[Proposition 8.6]{GetzlerKapranov}.
\begin{lem}
    The operation $\Exp(f)$ is invertible under plethysm, with inverse \[\Log(1 + f):= \sum_{n \geq 1} \frac{\mu(n)}{n}\log(1 + p_n \circ f).\]
    That is,
    \[ \Exp(\Log(1 + f)) = \Log(1 + \Exp(f)) = f, \]
    for any $f \in \hat{\Lambda}_+$.
\end{lem}

We note that plethysm extends to $\hat{\Lambda}[\![x ]\!]$: it defines a map
\[ \hat{\Lambda}[\![x ]\!] \times F_1 \hat{\Lambda }[\![x]\!] \to \hat{\Lambda}[\![ x ]\!], \]
where $F_1 \hat{\Lambda }[\![x]\!] \subset \hat{\Lambda }[\![x]\!]$ is the subspace defined by elements with vanishing bidegree-$(0, 0)$ term, by defining $x \circ f = x$ for any $f$ and $p_n \circ x = x^n$ for all $n > 0$.

\section{Proofs of the main theorems}
For $g \geq 2$ and $n > 0$ define $\calC_{g}^n \to \M_{g}$ be the $n$th fibered power of the universal curve over $\M_{g}$, and formally set $\calC_g^0 = \M_g$. The stack $\calC_{g}^n$ parameterizes smooth curves of genus $g$ with $n$ ordered, not-necessarily-distinct marked points, and is naturally identified with the stack $\M_{g, \ve^n}$ of smooth curves in Hassett's compactification $\Mbar_{g, \ve^n}$ of $\M_{g, n}$, for any $\ve \in \QQ_{> 0}$ with $\ve < 1/n$; this defines $\calC_g^n$ for $g = 1$ and $n > 0$ as well. We thus obtain the following lemma from \cite[Theorem A]{KSY2024}.
\begin{lem}
    For any $k$, we have
    \[ \cat{Z}_g^{k} \circ \Exp(p_1) = \sum_{n} \chi_k^{\bbS_n}(\calC_{g}^n). \]
    where the sum is over those $n$ as in (\ref{eqn:JandZdefn}).
\end{lem}
\begin{proof}
    Set
    \[\M_{g, m|n} := \M_{g, (1^m, \ve^n)} \subset \Mbar_{g, (1^m, \ve^n)} \]
    for the locus of smooth curves in the Hassett moduli space $\Mbar_{g, (1^m, \ve^n)}$. The precise definition of this moduli space is not important here, though it can be found in \cite{Hassett}. The important point is that when $m = 0$ we recover $\calC_g^n$.
    
    The weight-$k$ part of the formula for $\cat{a}_g$ in \cite[Theorem A]{KSY2024} reads
    \begin{equation}\label{eqn:KSYthm}
    \sum_{m,n} \chi_k^{\bbS_m \times \bbS_n}(\calM_{g, m|n}) = \Delta (\cat{Z}_g^k) \circ_2 \Exp(q_1),
    \end{equation}
    where the sum on the left-hand-side is over all $m, n \geq 0$ and must satisfy the additional constraint that $m + n>0$ if $g = 1$. We have again set
    \[ \hat{\Lambda} \otimes \hat{\Lambda} = \QQ[\![p_1, p_2, \ldots, q_1, q_2, \ldots  ]\!],\]
    and written $\circ_2$ for plethysm in the second set of variables as discussed in \S 2 of loc. cit. Now take the part of this formula which is of degree $0$ in the first set of variables. Formally, we are taking the projection
    \[ \hat{\Lambda} \otimes \hat{\Lambda} \to \QQ\otimes \hat{\Lambda} = \hat{\Lambda} \]
    which sends $p_i \to 0$ and $q_i \to p_i$ for all $i > 0$ on both sides of (\ref{eqn:KSYthm}). Using $\M_{g, 0|n} = \calC_g^n$, we obtain
    \[ \sum_{n} \chi_k^{\bbS_n}(\calC_g^n) = \cat{Z}_g^k \circ \Exp(p_1), \]
    as desired.    
\end{proof}
From (\ref{eqn:invDelta_space}) we obtain the following corollary. All sums over $m, n$ in this section are taken with the same convention as in (\ref{eqn:KSYthm}).
\begin{cor}\label{cor:light_univ_formula}
    For any $k$, we have
    \[ \mathrm{inv}\Delta(\cat{Z}_g^{k} \circ \Exp(p_1)) = \sum_{m,n} \chi_k^{\bbS_n}([\calC_{g}^{m + n}/\bbS_m])x^m. \]
\end{cor}

Recall that the projective bundle $\calS_{g, n}^m$ of (\ref{defn:Sgn}) is defined for $m \geq 2g - 1$, and its rational cohomology coincides with that of the stack quotient $[\calC_{g, n}^m / \bbS_m]$. We formally set
\[ \calS_{g, n}^m = [\calC_{g, n}^m/\bbS_m
] \]
for $0 \leq m \leq 2g - 2$.
The next lemma follows from the same reasoning as \cite[Proposition 4.3]{KSY2024}, or \cite[Theorem 3.2]{GetzlerPandharipande}: the operation $\circ \Exp(p_1)$ allows marked points to coincide.
\begin{lem}\label{lem:sym_to_light}
With the convention that $\calS_{g, n}^m = [\calC_{g, n}^m/\bbS_m]$ for $m \leq 2g - 2$,
    \[\left(\sum_{m, n} \chi_k^{\bbS_n}(\calS_{g, n}^m)x^m\right) \circ \Exp(p_1) =  \sum_{m,n} \chi_k^{\bbS_n}([\calC_{g}^{m + n}/\bbS_m])x^m.\]
\end{lem}
Since $\Log(1 + p_1)$ is the inverse of $\Exp(p_1)$ under plethysm, we obtain the following result by combining Corollary \ref{cor:light_univ_formula} and Lemma \ref{lem:sym_to_light}.
\begin{cor}\label{cor:Sym_formula}
With the convention that $\calS_{g, n}^m = [\calC_{g, n}^m/\bbS_m]$ for $m \leq 2g - 2$,
    \[\sum_{m,n} \chi^{\bbS_n}_k(\calS_{g, n}^m)x^m = \mathrm{inv}\Delta(\cat{Z}_g^{k} \circ \Exp(p_1)) \circ \Log(1 + p_1). \]
\end{cor}

\begin{rem}\label{rem:weight_to_top}
    The formula of Corollary \ref{cor:Sym_formula} holds verbatim if we replace $\chi_k^{\bbS_m}$ with the topological Euler characteristic $\chi^{\bbS_n}_{top}$ on both sides.
\end{rem}

\subsection{The fundamental transformation}
Recall from (\ref{eqn:Transform_defn}) the $\QQ$-algebra homomorphism
\[ \calT: \hat{\Lambda} \to \hat{\Lambda}[\![x]\!] \]
 by
\[p_k \mapsto \frac{p_k + x^k}{1-x^k}.  \]

\begin{lem}\label{lem:transf_formula}
    For any $f \in \hat{\Lambda}$, we have
    \[\calT(f) = \mathrm{inv}\Delta(f \circ \Exp(p_1)) \circ \Log(1 + p_1). \]
\end{lem}

\begin{proof}
    We use the fundamental symmetric function identity
    \[\sum_{n \geq 0} h_n = \exp\left(\sum_{n > 0} \frac{p_n}{n} \right) \]
    to find that
    \[\Exp(p_k) = \exp\left( \sum_{n > 0} \frac{p_{kn}}{n}\right) - 1  = p_k \circ \Exp(p_1).\]
Since \[\mathrm{inv}\Delta(p_j) = p_j + x^j \]
for any $j$, we find
\begin{align*}
         \mathrm{inv}\Delta(p_k \circ \Exp(p_1)) &= \exp\left( \sum_{n > 0} \frac{p_{kn} + x^{kn}}{n} \right) - 1
         \\&= \exp\left(\sum_{n > 0} \frac{p_{kn}}{n} \right)(1 - x^k)^{-1} - 1 \\&= (\Exp(p_k) + 1)(1 - x^{k})^{-1} - 1.
    \end{align*}
    Now simply observe that
    \[ \Exp(p_k) \circ \Log(1 + p_1) = p_k \circ \Exp(p_1) \circ \Log(1 + p_1) = p_k, \]
    so
    \[\mathrm{inv}\Delta(p_k \circ \Exp(p_1)) \circ \Log(1 + p_1) = (p_k + 1)(1- x^{k})^{-1 } - 1 = \frac{p_k + x^k}{1 -x^k} = \calT(p_k), \]
    as desired.
\end{proof}
We will now prove Theorem \ref{thm:general_transform}, which we restate for convenience.
\begin{customthm}{1}
For any $k$,
\[ \lim_{x \to 1}(1-x)\calT(\cat{Z}_g^k) = \sum_{i = 0}^{\floor{k/2}} \cat{J}_g^{k - 2i}. \]
The limit is interpreted as multiplying by $(1 -x)$ and then evaluating at $x= 1$.
\end{customthm}
\begin{proof}
From Corollary \ref{cor:Sym_formula} and Lemma \ref{lem:transf_formula}, we find
\[ \calT (\cat{Z}_g^k) = \sum_{m,n} \chi^{\bbS_n}_k(\calS_{g, n}^m)x^m \]
By Proposition \ref{prop:sym_to_pic_weights}, we have the formula
\[ \sum_{m,n}\chi_k^{\bbS_n}(\calS_{g, n}^m) = \sum_{n} \sum_{i = 0}^{m - g} \chi_{k - 2i}^{\bbS_n}(\Pic_{g, n}) = \sum_{i = 0}^{\floor{k/2}} \cat{J}_g^{k - 2i} \]
for $m \gg 0$. Therefore
\[ \calT(\cat{Z}_{g}^k) = P(x) + \sum_{i = 0}^{\floor{k/2}} \left(\cat{J}_g^{k - 2i}\right) \cdot\frac{x^N}{1-x} \]
for some $N > 0$ and a polynomial $P(x) \in \hat{\Lambda}[x]$. The desired statement follows immediately.
\end{proof}

We are now ready to prove Corollary \ref{cor:Jgformula}, whose statement we recall.
\begin{customcor}{2}
Suppose $g \geq 2$. The generating function $\cat{J}_g^0$ is obtained from the expression for $\cat{Z}_{g}^0$ in \cite[Theorem 1.1]{CFGP} by restricting the sum to those Laurent monomials in the $P_j$ whose exponents sum to $1$, and then making the substitution $P_j \mapsto P_j/j$ for all $j$. Specifically,
\[ \cat{J}_g^0 = \sum_{k, m, s, a, d} \frac{(-1)^{k}m^{k-1} (k-1)!}{P_m^k}  \prod_{i = 1}^{s} \frac{\mu(m/d_i)^{a_i} P_{d_i}^{a_i}}{d_i^{a_i}a_i!}. \]
where the sum is over integers $k, m, s > 0$ and $s$-tuples of positive integers $a = (a_1, \ldots, a_s)$ and $d = (d_1, \ldots, d_s)$ with $\gcd(d_1, \ldots, d_s) = 1$ and $d_i \mid m$ for each $i$, such that $0 < d_1 < \cdots < d_s < m$, $\sum_i a_i = k + 1$, and $\sum_i a_i d_i = km + 1 - g$.
\end{customcor}
\begin{proof}
 We analyze the substitution $P_k \mapsto P_k(1-x^k)^{-1}$ in the formula in \cite{CFGP} monomial-by-monomial. For any $\alpha \in \QQ$, this substitution takes
\[ \alpha \cdot \frac{1}{P_m^k} \prod_{i = 1}^{s} P_{d_i}^{a_i} \mapsto \alpha\cdot \frac{(1 - x^m)^k}{P_{m}^k} \prod_{i = 1}^s P_{d_i}^{a_i}(1 - x^{d_i})^{-a_i}. \]
Notice that the order of vanishing of this expression at $x= 1$ is
\[ k - \sum_{i = 1}^{s}a_i. \]
Thus if $k - \sum_{i = 1}^{s}a_i \geq 0$, we have
\[ \lim_{x \to 1} (1-x) \cdot \alpha\cdot \frac{(1 - x^m)^k}{P_{m}^k} \prod_{i = 1}^s P_{d_i}^{a_i}(1 - x^{d_i})^{-a_i} = 0. \]

When $k - \sum_{i = 1}^{s}a_i = -1$, we find
\[\lim_{x \to 1} (1 - x) \cdot \alpha\cdot \frac{(1 - x^m)^k}{P_{m}^k} \prod_{i = 1}^s P_{d_i}^{a_i}(1 - x^{d_i})^{-a_i} = \alpha\cdot \frac{m^k}{P_{m}^k} \prod_{i = 1}^s \frac{P_{d_i}^{a_i}}{d_i^{a_i}}. \]
This calculation follows directly from the factorization
\[ 1 - x^{\ell} = (1 - x)(1 + x + x^2 + \cdots + x^{\ell-1}). \]
Finally, notice that when $k -\sum_{i = 1}^{s}a_i = -1$, the product
\[ \prod_{p \mid (d_1, d_2, \ldots, d_s, m)} \left(1 - \frac{1}{p^r} \right) \]
appearing in \cite[Theorem 1.1]{CFGP} vanishes unless it is empty, when it is equal to $1$ (this is because $r = 1 + k - \sum_{i = 1}^{s}a_i$ in their formula). Thus only the terms with $\mathrm{gcd}(d_1, \ldots, d_s) = 1$ survive. Direct inspection of the formula in \cite{CFGP} reveals that no monomials with $k - \sum_{i = 1}^{s}a_i < -1$ appear, which justifies this monomial-by-monomial approach. 
\end{proof}
We are left to prove Theorem \ref{thm:top_chi}.

\begin{customthm}{3}\label{thm:top_chi}
Suppose $g \geq 2$. The generating function $\cat{J}_g^{top}$ is obtained from Gorsky's formula \cite[Theorem 3.8]{Gorsky} for $\cat{Z}_g^{top}$ by restricting the sum to Laurent monomials in the $P_j$ whose exponents sum to $2$, and then making the substitution $P_j \mapsto P_j/j$ for all $j$. Explicitly,
    \[ \cat{J}_g^{top} = \sum_{r = 2}^{4g +2} \sum_{s = 3}^{2g + 2} \sum_{\substack{{k_1 + \ldots + k_{r - 1} = s}\\ {\sum_i ik_i + r(2 - s)  = 2 - 2g}\\ {k_i \geq 0 \,\forall\, i}\\{i \mid r \text{ if }k_i > 0}\\{\mathrm{gcd}(i \mid k_i > 0) = 1}}}\left( (-1)^{s - 3}(s-3)! \frac{N(r;k_1, \ldots, k_{r-1})}{\prod_{i = 1}^{r - 1}i^{k_i}k_i!} \cdot r^{s - 3} \right) P_r^{2 - s} \prod_{i = 1}^{r - 1}P_{i}^{k_i}, \]
where
\[N(r; k_1, \ldots, k_{r - 1}) = \frac{1}{r} \sum_{d \mid r} \varphi(d) \prod_{i = 1}^{r-1} \left(\frac{\mu (d/\gcd(d, i)) \cdot\varphi(r/i)}{\varphi(d/\gcd(d, i))}\right)^{k_i}. \]
\end{customthm}
\begin{proof}
    We use that
    \[\calS_{g,n}^m \to \Pic_{g, n}^m \]
    is a $\P^{m-g}$-bundle for $m \geq 2g - 1$, together with the topological version of Corollary \ref{cor:Sym_formula} (cf. Remark \ref{rem:weight_to_top}) to find that there is a polynomial $Q(x) \in \hat{\Lambda}[x]$ such that
    \begin{align*}
        \calT(\cat{Z}_g^{top})& = Q(x) + \cat{J}_g^{top} \sum_{m \geq 2g - 1} (m - g + 1) x^m  \\&= Q(x) +  \cat{J}_g^{top} \cdot \frac{x^{2g-1}\big(g-(g-1)x\big)}{(1-x)^2}.
    \end{align*}
    Thus
    \[ \lim_{x \to 1} (1-x)^2\calT(\cat{Z}_g^{top}) = \cat{J}_g^{top}.  \]
    Using the same logic as in the proof of Corollary \ref{cor:Jgformula}, we find that a monomial $M$ in the $P_i$'s must have its sum of exponents equal to $2$ for the limit
    \[ \lim_{x \to 1}(1-x)^2M|_{P_k \to P_k(1 - x^k)^{-1}\,\forall \,k} \]
    to converge to a nonzero value. Also by the same reasoning as in the proof of Corollary \ref{cor:Jgformula}, we find that taking the limit for such monomials amounts to the substitution $P_j \mapsto P_j/j$. One also finds from equation (3) of the proof of \cite[Proposition 3.1]{Gorsky} that the sum of the exponents in Gorsky's formula is always bounded above by $2$; this justifies the outlined procedure for taking the limit. The condition on the greatest common divisors of the $i$ such that $k_i > 0$ comes from the product
    \[ \prod_{p \mid \gcd(i \mid k_i > 0)} \left( 1 - \frac{1}{p^{2h}}\right) \]
    in Gorsky's formula, for a parameter $h$ which equals $0$ when the exponents sum to $2$. The factor $(-1)^{s-3}(s - 3)!$ corresponds to $\chi^{orb}(\M_{0, s})$ in his formula, and the conditions on $k_i$ appearing in the sum are implicit in his definition of the constants $c_{k_1, \ldots, k_r}$.
\end{proof}

\begin{rem}
    There are superficial differences between our formula and Gorsky's \cite[Theorem 3.8]{Gorsky}. There is one very minor typo in Gorsky's formula: the definition of $h$ should read
    \[ h = 1 - \frac{1}{2} \sum_{i = 1}^{r - 1} k_i, \]
    instead of $(1/2) \cdot(1 - \sum_{i = 1}^{r-1}k_i)$. Also, Gorsky uses the notation $N(r;\lambda)$ for what we denote as $N(r;k_1, \ldots, k_{r - 1})$, where $\lambda = (\lambda_1, \ldots, \lambda_{s})$ is the partition 
    \[\lambda = \left(\underbrace{1, \ldots, 1}_{k_1}, \underbrace{2, \ldots, 2}_{k_2}, \ldots, \underbrace{r-1, \ldots, r-1}_{k_{r - 1}}\right). \]
    In \cite[Definition 3.5]{Gorsky}, he defines
    \[ N(r;\lambda) :=\left| \left\{(x_1, x_2, \ldots, x_s) \in (\Z/r\Z)^s \mid \sum_{i = 1}^{s} x_i \equiv 0 \pmod{r},\quad \gcd(k, x_i) = \lambda_i  \,\forall i  \right\}\right|; \]
    in this definition $k$ should be replaced by $r$. Finally, having defined
    \[ c(k, \ell, d) := \mu\left( \frac{d}{\gcd(d, \ell)}\right) \cdot\frac{\varphi(k/\ell)}{\varphi(d/\gcd(d, \ell))}, \] his formula for $N(r, \lambda)$ in \cite[Lemma 3.6]{Gorsky} reads
    \[N(r; \lambda) = \frac{1}{r}\sum_{d \mid r}\varphi(d)\prod_{i = 1}^s c(k, \lambda_i, d),\]
    and $k$ should read as $r$. After these substitutions, our formula is extracted directly from his by the transformation defined in Theorem \ref{thm:top_chi}.
\end{rem}

\begin{table}[h]
\centering
\renewcommand{\arraystretch}{2}
\begin{tabular}{c l}
\toprule
$g$ & \multicolumn{1}{c}{$\cat{J}_g^{0}$} \\\midrule

2 &
$\displaystyle
-\frac{1}{6}\frac{P_2 P_3}{P_6}
-\frac{1}{2}\frac{P_1^2}{P_3}
-\frac{1}{3}\frac{P_1^3}{P_2^2}$ \\
\midrule
3 &
$\displaystyle
\frac{1}{3}\frac{P_1 P_3}{P_6}
-\frac{1}{3}\frac{P_1^4}{P_2^3}$ \\
\midrule
4 &
$\displaystyle
-\frac{1}{10}\frac{P_2 P_5}{P_{10}}
+\frac{1}{2}\frac{P_1 P_2}{P_6}
-\frac{1}{2}\frac{P_1^2}{P_5}
-\frac{1}{2}\frac{P_1^3}{P_3^2}
-\frac{2}{5}\frac{P_1^5}{P_2^4}$ \\
\midrule
5 &
$\displaystyle
\frac{1}{5}\frac{P_1 P_5}{P_{10}}
-\frac{1}{6}\frac{P_2 P_3^2}{P_6^2}
-\frac{1}{2}\frac{P_1^2}{P_6}
-\frac{8}{15}\frac{P_1^6}{P_2^5}$ \\
\midrule
6 &
$\displaystyle
-\frac{1}{14}\frac{P_2 P_7}{P_{14}}
-\frac{1}{2}\frac{P_1^2}{P_7}
-\frac{1}{4}\frac{P_2^2 P_3}{P_6^2}
+\frac{1}{3}\frac{P_1 P_3^2}{P_6^2}
-\frac{3}{4}\frac{P_1^4}{P_3^3}
-\frac{16}{21}\frac{P_1^7}{P_2^6}$ \\
\midrule
7 &
$\displaystyle
\frac{1}{7}\frac{P_1 P_7}{P_{14}}
+\frac{P_1 P_2 P_3}{P_6^2}
-\frac{8}{7}\frac{P_1^8}{P_2^7}$ \\
\midrule
8 &
$\displaystyle
-\frac{1}{15}\frac{P_3 P_5}{P_{15}}
+\frac{1}{2}\frac{P_1 P_2}{P_{10}}
-\frac{2}{9}\frac{P_2 P_3^3}{P_6^3}
+\frac{3}{4}\frac{P_1 P_2^2}{P_6^2}
-\frac{P_1^2 P_3}{P_6^2}
-\frac{5}{6}\frac{P_1^3}{P_5^2}
-\frac{27}{20}\frac{P_1^5}{P_3^4}
-\frac{16}{9}\frac{P_1^9}{P_2^8}$ \\
\midrule
9 &
$\displaystyle
-\frac{1}{10}\frac{P_2 P_5^{2}}{P_{10}^{2}}
-\frac{1}{2}\frac{P_1^{2}}{P_{10}}
-\frac{1}{2}\frac{P_2^{2} P_3^{2}}{P_6^{3}}
+\frac{4}{9}\frac{P_1 P_3^{3}}{P_6^{3}}
-\frac{3}{2}\frac{P_1^{2} P_2}{P_6^{2}}
-\frac{128}{45}\frac{P_1^{10}}{P_2^{9}}$ \\

\bottomrule
\end{tabular}
\caption{$\cat{J}_g^0$ for $2 \leq g \leq 9$}
\label{table:wt0data}
\end{table}

\begin{table}[ht]
\centering
\renewcommand{\arraystretch}{1}
\begin{tabular}{@{}c p{0.82\textwidth}@{}}
\toprule
$g$ & \multicolumn{1}{c}{$\cat{J}_g^{top}$} \\
\midrule
$2$ &
\[
\begin{aligned}
\frac{2}{5}\,\frac{P_1 P_2 P_5}{P_{10}}
&+ \frac{1}{2}\,\frac{P_1^{2} P_4}{P_8}
- \frac{1}{12}\,\frac{P_2^{2} P_3^{2}}{P_6^{2}}
+ \frac{1}{2}\,\frac{P_1^{2} P_2}{P_6} \\[4pt]
&+ 2\,\frac{P_1^{3}}{P_5}
- \frac{1}{2}\,\frac{P_1^{2} P_2^{2}}{P_4^{2}}
- \frac{3}{4}\,\frac{P_1^{4}}{P_3^{2}}
- \frac{1}{15}\,\frac{P_1^{6}}{P_2^{4}}
\end{aligned}
\]
\\
\midrule
$3$ &
\[
\begin{aligned}
\frac{3}{7}\,\frac{P_1 P_2 P_7}{P_{14}}
&+ \frac{1}{3}\,\frac{P_1 P_3 P_4}{P_{12}}
+ \frac{1}{3}\,\frac{P_1^{2} P_6}{P_{12}}
+ 2\,\frac{P_1^{2} P_3}{P_9} \\[4pt]
&+ 2\,\frac{P_1^{2} P_2}{P_8}
+ 5\,\frac{P_1^{3}}{P_7}
- \frac{1}{2}\,\frac{P_1 P_2^{2} P_3}{P_6^{2}}
- \frac{1}{3}\,\frac{P_1^{2} P_3^{2}}{P_6^{2}} \\[4pt]
&+ \frac{2}{3}\,\frac{P_1^{2} P_2^{3}}{P_4^{3}}
- \frac{4}{3}\,\frac{P_1^{4}}{P_4^{2}}
+ \frac{3}{2}\,\frac{P_1^{5}}{P_3^{3}}
- \frac{2}{21}\,\frac{P_1^{8}}{P_2^{6}}
\end{aligned}
\]
\\
\midrule
$4$ &
\[
\begin{aligned}
\frac{1}{3}\,\frac{P_1 P_2 P_9}{P_{18}}
&+ \frac{1}{2}\,\frac{P_1^{2} P_8}{P_{16}}
+ \frac{8}{15}\,\frac{P_1 P_3 P_5}{P_{15}}
+ \frac{2}{3}\,\frac{P_1 P_2 P_3}{P_{12}} \\[4pt]
&+ \frac{1}{2}\,\frac{P_1^{2} P_4}{P_{12}}
- \frac{1}{10}\,\frac{P_2^{2} P_5^{2}}{P_{10}^{2}}
+ 3\,\frac{P_1^{2} P_2}{P_{10}}
+ 3\,\frac{P_1^{3}}{P_9} \\[4pt]
&- \frac{1}{2}\,\frac{P_1^{2} P_4^{2}}{P_8^{2}}
+ \frac{1}{6}\,\frac{P_2^{3} P_3^{2}}{P_6^{3}}
+ \frac{4}{9}\,\frac{P_1 P_2 P_3^{3}}{P_6^{3}}
- \frac{9}{4}\,\frac{P_1^{2} P_2^{2}}{P_6^{2}} \\[4pt]
&- \frac{2}{3}\,\frac{P_1^{3} P_3}{P_6^{2}}
- \frac{65}{6}\,\frac{P_1^{4}}{P_5^{2}}
- \frac{P_1^{2} P_2^{4}}{P_4^{4}}
+ \frac{16}{3}\,\frac{P_1^{4} P_2}{P_4^{3}} \\[4pt]
&- \frac{99}{20}\,\frac{P_1^{6}}{P_3^{4}}
- \frac{8}{45}\,\frac{P_1^{10}}{P_2^{8}}
\end{aligned}
\]
\\
\bottomrule
\end{tabular}
\caption{$\cat{J}_g^{top}$ for $g = 2, 3, 4$}
\label{table:Jgtop}
\end{table}

\bibliographystyle{amsalpha}
\clearpage
\bibliography{reference}

\end{document}